\documentclass[11pt]{amsart}

\usepackage[T1]{fontenc}
\usepackage{lmodern}
\usepackage{amsmath,amssymb,amsthm,mathtools,mathrsfs}
\usepackage{booktabs}
\usepackage{enumitem}
\usepackage{microtype}
\usepackage{aliascnt}
\usepackage[hidelinks]{hyperref}

\allowdisplaybreaks
\numberwithin{equation}{section}
\setlist[enumerate]{leftmargin=*,itemsep=0.25em,topsep=0.4em}

\newtheorem{theorem}{Theorem}[section]
\newaliascnt{lemma}{theorem}
\newtheorem{lemma}[lemma]{Lemma}
\aliascntresetthe{lemma}
\newaliascnt{proposition}{theorem}
\newtheorem{proposition}[proposition]{Proposition}
\aliascntresetthe{proposition}
\newaliascnt{corollary}{theorem}
\newtheorem{corollary}[corollary]{Corollary}
\aliascntresetthe{corollary}
\newaliascnt{claim}{theorem}

\aliascntresetthe{claim}
\theoremstyle{definition}
\newaliascnt{definition}{theorem}

\aliascntresetthe{definition}
\newaliascnt{example}{theorem}

\aliascntresetthe{example}
\theoremstyle{remark}
\newaliascnt{remark}{theorem}
\newtheorem{remark}[remark]{Remark}
\aliascntresetthe{remark}

\usepackage[capitalise,nameinlink,noabbrev]{cleveref}
\crefname{theorem}{Theorem}{Theorems}
\crefname{lemma}{Lemma}{Lemmas}
\crefname{proposition}{Proposition}{Propositions}
\crefname{corollary}{Corollary}{Corollaries}
\crefname{claim}{Claim}{Claims}
\crefname{definition}{Definition}{Definitions}
\crefname{example}{Example}{Examples}
\crefname{remark}{Remark}{Remarks}

\newcommand{\V}{\mathsf V}
\newcommand{\Id}{\operatorname{Id}}

\newcommand{\M}{\mathbf M}

\newcommand{\NF}{\mathbf{NF}}

\newcommand{\Occ}{\operatorname{occ}}

\newcommand{\op}{\mathrm{op}}
\newcommand{\Mat}{\operatorname{Mat}}
\newcommand{\Core}{\operatorname{Core}}
\makeatletter
\renewcommand{\subjclass}[2][2020]{%
  \gdef\@subjclass{#1 \emph{Mathematics Subject Classification}. #2}%
}
\makeatother

\title[Matrix varieties over $S_7$]
{Matrix varieties over $S_7$:\\
dimension-three stability and continuum-sized subvariety intervals}

\author{Jun Jiao}
\author{Xiaolei Shao}



\begin{document}

\begin{abstract}
We study matrix semirings over the three-element flat additively idempotent
semiring with elements one, a, and infinity, where the square of a is
infinity.  We determine the associated matrix-variety chain completely.  The
scalar, two-by-two, and three-by-three cases generate three distinct
varieties, while every matrix dimension at least three generates the same
variety.  The proof shows that every failure of an identity in arbitrary
dimension is already witnessed on three indices, and it yields a coordinate
criterion for all identities in the stable variety.  We also realize every
graph semiring arising from a directed graph of in-degree and out-degree at
most one as a divisor of a direct power of the two-by-two matrix semiring.
Consequently, the variety generated by all three-nilpotent flat semirings is
contained in the two-by-two matrix variety.  Directed cycles and independent
reversal identities then embed the power-set lattice of the odd primes into
each interval between the base variety and a nontrivial matrix variety.
Thus every such interval has continuum cardinality and contains
continuum-sized chains and antichains.  Finally, we determine the last three
powers of the multiplicative subsemiring obtained by deleting the constant
all-one matrix, and we develop general matrix operators on the lattice of
additively idempotent semiring varieties, including stable closures, stable
cores, and propagation of equality along matrix-dimension chains.
\end{abstract}

\maketitle

\section{Introduction}

An \emph{additively idempotent semiring}, or \emph{ai-semiring}, is an
algebra $(S,+,\cdot)$ whose additive reduct is a commutative idempotent
semigroup, whose multiplicative reduct is a semigroup, and in which
multiplication distributes over addition on both sides.  If $S$ is an
ai-semiring, then $\V(S)$ denotes the variety generated by $S$, and
$\Id(S)$ denotes its equational theory.

The present paper concerns the three-element flat ai-semiring
\[
 S_{7}=\{1,a,\infty\},
\]
with operations
\[
\begin{array}{c|ccc}
 +&1&a&\infty\\ \hline
 1&1&\infty&\infty\\
 a&\infty&a&\infty\\
 \infty&\infty&\infty&\infty
\end{array}
\qquad
\begin{array}{c|ccc}
 \cdot&1&a&\infty\\ \hline
 1&1&a&\infty\\
 a&a&\infty&\infty\\
 \infty&\infty&\infty&\infty.
\end{array}
\]
Thus $\infty$ is both the additive maximum and the multiplicative zero, and
$a^{2}=\infty$.  The semiring $S_{7}$ is the unique nonfinitely based
three-element ai-semiring up to isomorphism, and its position in the lattice
of ai-semiring varieties has generated a substantial finite-basis and
subvariety-lattice theory; see, in particular, \cite{JRZ,GJRZ}.

For an ai-semiring $S$, let $\M_{n}(S)$ denote the semiring of all $n$ by $n$
matrices over $S$, equipped with the usual matrix operations.  Jiao and Ren
proved that $\M_{n}(S)$ embeds into $\M_{n+1}(S)$ for every ai-semiring $S$
and every positive matrix dimension.  Applied to $S_{7}$, this gives an
ascending chain
\[
 \V(S_{7})<\V(\M_{2}(S_{7}))
 \leq\V(\M_{3}(S_{7}))\leq\cdots.
\]
They further proved that every variety in
$[\V(S_{7}),\V(\M_{n}(S_{7}))]$ is nonfinitely based, that this interval
contains at least countably many distinct varieties, and that
$\M_{n}(S_{7})\setminus\{[1]_{n}\}$ is multiplicatively $5$-nilpotent but
not $4$-nilpotent; see \cite{JR}.  Two natural questions remained: whether
the matrix-variety chain stabilizes, and whether the displayed intervals are
uncountable.

The first main result gives an exact answer to the stability question.

\begin{theorem}[Dimension-three stability]\label{thm:intro-stability}
For every $n\geq3$,
\[
 \Id(\M_{n}(S_{7}))=\Id(\M_{3}(S_{7})).
\]
Equivalently,
\[
 \V(\M_{n}(S_{7}))=\V(\M_{3}(S_{7})).
\]
Moreover,
\[
 \V(S_{7})<\V(\M_{2}(S_{7}))<\V(\M_{3}(S_{7})).
\]
Consequently, the matrix-variety chain over $S_{7}$ has exactly three
distinct members.
\end{theorem}

The equality from dimension three onward is obtained by a three-index
principle.  We prove that whenever an identity fails in an arbitrary matrix
dimension, one can retain at most three matrix indices and obtain a failure
in $\M_{3}(S_{7})$.  The strict inequality between dimensions two and three
is witnessed by the identity
\begin{equation}\label{eq:separating-intro}
 x+y+xy+yx\ \approx\ x+y+xy+yx+x^{2}.
\end{equation}
It holds in $\M_{2}(S_{7})$ and fails in $\M_{3}(S_{7})$.  Combining the
three-index theorem with the known content--$\delta$ criterion for $S_{7}$
from \cite{GJRZ}, we also obtain a complete coordinate criterion for the
identities of the stable variety $\V(\M_{3}(S_{7}))$.

The second main result settles the cardinality question in the strongest
possible form.

\begin{theorem}[Continuum subvariety intervals]\label{thm:intro-continuum}
For every $n\geq2$,
\[
 \bigl|[\V(S_{7}),\V(\M_{n}(S_{7}))]\bigr|=2^{\aleph_{0}}.
\]
Each of these intervals contains a chain and an antichain of cardinality
$2^{\aleph_{0}}$.  Every variety in every such interval is nonfinitely based.
\end{theorem}

The proof proceeds through graph semirings.  Gao and Ren characterized the
nontrivial subdirectly irreducible members of the variety $\NF_{3}$ generated
by all $3$-nilpotent flat semirings: they are precisely the graph semirings
associated with directed graphs whose in-degrees and out-degrees are at most
one \cite{GR}.  We construct, for every such graph $\mathbb G$, an explicit
subsemiring of a direct power of $\M_{2}(S_{7})$ and an explicit congruence
whose quotient is the graph semiring $S_{\mathbb G}$.  It follows that
\[
 \NF_{3}\leq\V(\M_{2}(S_{7})).
\]
Directed cycle graph semirings, together with forward--reverse cycle
identities, then yield an order embedding of the power-set lattice of the odd
primes into $[\V(S_{7}),\V(\M_{2}(S_{7}))]$.

A third family of results refines the known nilpotency theorem.  Put
\[
 N_{n}=\M_{n}(S_{7})\setminus\{[1]_{n}\}.
\]
For nonempty $I,J\subseteq[n]$, let $E_{I,J}$ be the matrix having entry $a$
on $I\times J$ and entry $\infty$ elsewhere.  We prove
\[
\begin{aligned}
 N_{n}^{3}
   &=\{[\infty]_{n}\}\cup
     \{E_{I,J}:\varnothing\neq I,J\subseteq[n],
                   (I,J)\neq([n],[n])\},\\
 N_{n}^{4}
   &=\{[\infty]_{n}\}\cup
     \{E_{I,J}:\varnothing\neq I,J\subsetneq[n]\},\\
 N_{n}^{5}&=\{[\infty]_{n}\}.
\end{aligned}
\]
Consequently,
\[
 |N_{n}^{3}|=(2^{n}-1)^{2},\qquad
 |N_{n}^{4}|=(2^{n}-2)^{2}+1.
\]
Thus the last nonzero powers are described completely, rather than only
through their nilpotency index.

The final part of the paper places these concrete results in a general
variety-theoretic framework.  For every positive integer $n$, we introduce
the operator
\[
 \mathscr M_{n}(\mathcal V)
 =\V\{\M_{n}(S):S\in\mathcal V\}
\]
on the lattice of ai-semiring varieties.  These operators compose according
to multiplication of positive integers, preserve arbitrary joins, and are
monotone in the matrix dimension.  We prove a two-by-two test for matrix
stability, construct the least matrix-stable variety above a given variety
and the greatest matrix-stable variety below it, and show that one equality
at two distinct matrix dimensions forces eventual equality at all
sufficiently large dimensions.

The paper is organized as follows.  \Cref{sec:preliminaries} records the
matrix coordinate expansion and a strengthened matrix embedding.
\Cref{sec:three-index} proves dimension-three stability.  The exact chain is
determined in \cref{sec:exact-chain}, and the coordinate criterion for the
stable equational theory is obtained in \cref{sec:identity-criterion}.
Graph semirings and continuum many intermediate varieties are treated in
\cref{sec:graph-semirings,sec:continuum}.  The exact terminal powers of
$N_{n}$ are computed in \cref{sec:terminal-powers}.  Finally,
\cref{sec:matrix-operators} develops the general matrix-operator theory.

\section{Preliminaries and matrix embeddings}\label{sec:preliminaries}

An ai-semiring term is identified with a finite nonempty set of nonempty
words.  We write such a term as a formal sum
\[
 \mathbf u=u_{1}+\cdots+u_{r}.
\]
Repeated summands are immaterial because addition is idempotent.

For matrices $A_{1},\dots,A_{m}\in\M_{n}(S_{7})$ and indices $i,j\in[n]$,
\begin{equation}\label{eq:path-expansion}
 (A_{1}\cdots A_{m})_{ij}
 =\sum_{i_{1},\dots,i_{m-1}\in[n]}
   (A_{1})_{i i_{1}}(A_{2})_{i_{1}i_{2}}\cdots
   (A_{m})_{i_{m-1}j}.
\end{equation}
We call a sequence
\[
 i=i_{0},i_{1},\dots,i_{m-1},i_{m}=j
\]
an index sequence for the coordinate $(i,j)$.

We shall use the following elementary consequence of flatness repeatedly.

\begin{lemma}\label{lem:flat-sum}
Let $z_{1},\dots,z_{r}\in S_{7}$ and let $c\in\{1,a\}$.  Then
\[
 z_{1}+\cdots+z_{r}=c
\]
if and only if $z_{1}=\cdots=z_{r}=c$.
Consequently, a product in $S_{7}$ is equal to $1$ if and only if all its
factors are $1$, and it is equal to $a$ if and only if exactly one factor is
$a$ and all remaining factors are $1$.
\end{lemma}

\begin{proof}
The assertion about sums follows directly from flatness: two distinct finite
elements sum to $\infty$, and any sum containing $\infty$ is $\infty$.
The assertion about products follows from the multiplication table and
$a^{2}=\infty$.
\end{proof}

We first strengthen the usual consecutive-dimension embedding.

\begin{theorem}[Surjective-index embedding]\label{thm:surjective-embedding}
Let $S$ be an ai-semiring, let $m\geq n\geq1$, and let
$f:[m]\twoheadrightarrow[n]$ be a surjection.  Define
\[
 \Phi_{f}:\M_{n}(S)\longrightarrow\M_{m}(S),\qquad
 \bigl(\Phi_{f}(A)\bigr)_{rs}=A_{f(r),f(s)}.
\]
Then $\Phi_{f}$ is an injective semiring homomorphism.
\end{theorem}

\begin{proof}
It is immediate that $\Phi_{f}$ preserves addition.  Let
$A,B\in\M_{n}(S)$.  For $r,s\in[m]$,
\begin{align*}
 (\Phi_{f}(A)\Phi_{f}(B))_{rs}
 &=\sum_{k=1}^{m}A_{f(r),f(k)}B_{f(k),f(s)}\\
 &=\sum_{t=1}^{n}\ \sum_{k\in f^{-1}(t)}
       A_{f(r),t}B_{t,f(s)}.
\end{align*}
Every fibre $f^{-1}(t)$ is nonempty.  The inner sum consists of one or more
copies of the same element, and idempotence of addition therefore gives
\[
 (\Phi_{f}(A)\Phi_{f}(B))_{rs}
 =\sum_{t=1}^{n}A_{f(r),t}B_{t,f(s)}
 =(AB)_{f(r),f(s)}
 =\bigl(\Phi_{f}(AB)\bigr)_{rs}.
\]
Thus $\Phi_{f}$ preserves multiplication.  Since $f$ is onto, one may choose
$r_{i}\in f^{-1}(i)$ for each $i\in[n]$ and recover
$A_{ij}$ as $\bigl(\Phi_{f}(A)\bigr)_{r_{i}r_{j}}$.  Hence $\Phi_{f}$ is
injective.
\end{proof}

\begin{corollary}\label{cor:dimension-chain}
For every ai-semiring $S$ and all $m\geq n\geq1$,
\[
 \M_{n}(S)\hookrightarrow\M_{m}(S),
 \qquad
 \V(\M_{n}(S))\leq\V(\M_{m}(S)).
\]
In particular, the map $s\mapsto[s]_{n}$ embeds $S$ into $\M_{n}(S)$.
\end{corollary}

\section{A three-index theorem for identities}\label{sec:three-index}

The purpose of this section is to establish
\[
 \Id(\M_{n}(S_{7}))=\Id(\M_{3}(S_{7}))\qquad(n\geq3).
\]
The proof uses only the coordinate expansion \eqref{eq:path-expansion} and the
fact that a finite value in $S_{7}$ forces every summand in that expansion to
have the same value.

\subsection{The matrix positions used by a word occurrence}

Fix a coordinate $(i,j)$ and a word
\[
 w=x_{1}x_{2}\cdots x_{m}.
\]
At its $t$-th position, the entry of the matrix assigned to $x_{t}$ that may
occur in \eqref{eq:path-expansion} has one of the following forms:
\begin{equation}\label{eq:occurrence-types}
\begin{array}{c|c}
\text{position of the occurrence}&\text{possible matrix entries}\\ \hline
m=1&(i,j)\\
 m\geq2,\ t=1&(i,s),\ s\in[n]\\
 m\geq2,\ t=m&(r,j),\ r\in[n]\\
 1<t<m&(r,s),\ r,s\in[n].
\end{array}
\end{equation}
We refer only to these four possibilities and do not attach any additional
terminology to them.  In what follows, ``first'' and ``last'' always refer to
an occurrence in a word of length at least two; a one-letter word is treated
separately.

\begin{lemma}\label{lem:missing-position}
Let $\mathbf u$ be an ai-semiring term and let
$q=x_{1}\cdots x_{m}$ be a word.  Suppose that one of the following occurs.
\begin{enumerate}[label=\textup{(\roman*)}]
\item $m=1$, and $x_{1}$ does not occur in any word of $\mathbf u$;
\item $m\geq2$, and $x_{1}$ occurs neither as the first letter nor as an
      interior letter of any word of $\mathbf u$;
\item $m\geq2$, and $x_{m}$ occurs neither as the last letter nor as an
      interior letter of any word of $\mathbf u$;
\item for some $1<t<m$, the letter $x_{t}$ has no interior occurrence in any
      word of $\mathbf u$.
\end{enumerate}
Then the identity
\[
 \mathbf u\approx\mathbf u+q
\]
fails in $\M_{3}(S_{7})$.
\end{lemma}

\begin{proof}
We evaluate at the coordinate $(1,2)$.  Assign the constant matrix $[1]_{3}$
to every variable, except that one entry of the matrix $X$ assigned to the
letter singled out in the relevant case is changed from $1$ to $a$:
\[
\begin{array}{c|c}
\text{case}&\text{changed entry}\\ \hline
\textup{(i)}&X_{12}\\
\textup{(ii)}&X_{13}\\
\textup{(iii)}&X_{32}\\
\textup{(iv)}&X_{33}.
\end{array}
\]
In case \textup{(i)}, the variable does not occur in $\mathbf u$.  In case
\textup{(ii)}, every occurrence of the variable in $\mathbf u$ is either a
last occurrence, which uses only column $2$, or a one-letter occurrence,
which uses only $(1,2)$.  Thus no word of $\mathbf u$ can use $(1,3)$.  The
argument in case \textup{(iii)} is symmetric.  In case \textup{(iv)}, first
occurrences use row $1$, last occurrences use column $2$, and one-letter
occurrences use $(1,2)$, so none can use $(3,3)$.

It follows that every index sequence for every word in $\mathbf u$ uses only
entries equal to $1$.  Hence
\[
 \mathbf u_{12}=1.
\]
On the other hand, $q$ has an index sequence that uses the changed entry.  Its
value at $(1,2)$ is therefore either $a$ or $\infty$, but not $1$.  Thus
$(\mathbf u+q)_{12}\neq\mathbf u_{12}$.
\end{proof}

\subsection{Words whose value is \texorpdfstring{$a$}{a}}

\begin{lemma}[Layer structure]\label{lem:layer-structure}
Let $A_{1},\dots,A_{m}\in\M_{n}(S_{7})$ and suppose that
\[
 (A_{1}\cdots A_{m})_{ij}=a.
\]
Put
\[
 V_{0}=\{i\},\qquad V_{m}=\{j\},\qquad
 V_{t}=[n]\quad(1\leq t<m).
\]
Then there exist maps
\[
 \lambda_{t}:V_{t}\longrightarrow\{0,1\}\qquad(0\leq t\leq m)
\]
such that
\[
 \lambda_{0}(i)=0,\qquad \lambda_{m}(j)=1,
\]
and, for every $1\leq t\leq m$, $r\in V_{t-1}$ and $s\in V_{t}$,
\begin{equation}\label{eq:lambda-edge}
 (A_{t})_{rs}=
 \begin{cases}
  1,&\lambda_{t}(s)=\lambda_{t-1}(r),\\
  a,&\lambda_{t}(s)=\lambda_{t-1}(r)+1.
 \end{cases}
\end{equation}
In particular, no entry occurring in the indicated ranges is $\infty$.
Furthermore, at most one of the internal maps
$\lambda_{1},\dots,\lambda_{m-1}$ is nonconstant.
\end{lemma}

\begin{proof}
By \cref{lem:flat-sum}, every index sequence from $i$ to $j$ in
\eqref{eq:path-expansion} has product $a$.  Hence every such sequence contains
exactly one entry equal to $a$, and all its other entries are $1$.

Fix $t$ and $s\in V_{t}$.  Choose an arbitrary index sequence from $i$ to
$s$ through the first $t$ matrices and define $\lambda_{t}(s)$ to be the
number of entries equal to $a$ along this partial sequence.  This number is
independent of the chosen partial sequence.  Indeed, if two partial sequences
ending at $s$ contained different numbers of $a$'s, append the same arbitrary
sequence from $s$ to $j$.  The two resulting complete index sequences would
contain different numbers of $a$'s, although both must contain exactly one.

Every partial sequence can be extended to a complete one, so
$\lambda_{t}(s)\in\{0,1\}$.  Appending the entry $(A_{t})_{rs}$ to a partial
sequence ending at $r$ gives
\[
 \lambda_{t}(s)=\lambda_{t-1}(r)+
 \begin{cases}
 0,&(A_{t})_{rs}=1,\\
 1,&(A_{t})_{rs}=a.
 \end{cases}
\]
This proves \eqref{eq:lambda-edge} and also excludes $\infty$.

If $\lambda_{t-1}$ takes the value $1$ somewhere, then
\eqref{eq:lambda-edge} forces $\lambda_{t}$ to be identically $1$, since every
state in $V_{t-1}$ is connected to every state in $V_t$ in the matrix
expansion.  Dually, if $\lambda_{t}$ takes the value $0$ somewhere, then
$\lambda_{t-1}$ is identically $0$.  Therefore, if some internal
$\lambda_{t}$ is nonconstant, all earlier layers are identically $0$ and all
later layers are identically $1$.  There can be at most one such layer.
\end{proof}

\begin{corollary}\label{cor:one-index-dependence}
Under the assumptions of \cref{lem:layer-structure}, define
\[
 \eta_{t}(r,s)=
 \begin{cases}
 0,&(A_{t})_{rs}=1,\\
 1,&(A_{t})_{rs}=a.
 \end{cases}
\]
On the range of entries used at position $t$, the function $\eta_{t}$ is one
of the following:
\[
 \text{a constant},\qquad \rho_{t}(r),\qquad \tau_{t}(s),
\]
where $\rho_{t}$ depends only on the row index and $\tau_{t}$ only on the
column index.
\end{corollary}

\begin{proof}
By \eqref{eq:lambda-edge},
\[
 \eta_{t}(r,s)=\lambda_{t}(s)-\lambda_{t-1}(r).
\]
If both adjacent layer maps are constant, then $\eta_{t}$ is constant.  If
$\lambda_{t}$ is the unique nonconstant layer, then
$\lambda_{t-1}\equiv0$ and $\eta_{t}$ depends only on $s$.  If
$\lambda_{t-1}$ is the unique nonconstant layer, then
$\lambda_{t}\equiv1$ and $\eta_{t}$ depends only on $r$.
\end{proof}

\subsection{Three indices are sufficient}

\begin{lemma}\label{lem:separated-count}
Let $\mathbf u$ be a term, let $q=x_{1}\cdots x_{m}$ be a word, and suppose
that under a matrix assignment in $\M_{n}(S_{7})$ one has
\[
 \mathbf u_{ij}=a.
\]
Assume that none of the four alternatives in \cref{lem:missing-position}
occurs.  For an index sequence
\[
 i=s_{0},s_{1},\dots,s_{m-1},s_{m}=j
\]
let $N(s_{1},\dots,s_{m-1})$ be the number of entries equal to $a$ in the
corresponding product for $q$.  Then all entries used by $q$ belong to
$\{1,a\}$, and there exist a nonnegative integer $C$ and maps
\[
 h_{t}:[n]\longrightarrow\{0,1,2\}\qquad(1\leq t<m)
\]
such that
\begin{equation}\label{eq:separated-count}
 N(s_{1},\dots,s_{m-1})=C+\sum_{t=1}^{m-1}h_{t}(s_{t}).
\end{equation}
\end{lemma}

\begin{proof}
Since $\mathbf u_{ij}=a$, every word occurring in $\mathbf u$ has value $a$
at $(i,j)$, and every index sequence for every such word has product $a$.

Consider the occurrence of $x_{t}$ at position $t$ in $q$.  By the assumption,
there is an occurrence of the same variable in a word of $\mathbf u$ whose
range of possible matrix entries contains the range used at position $t$ of
$q$: for a first occurrence we use a first or interior occurrence in
$\mathbf u$; for a last occurrence, a last or interior occurrence; for an
interior occurrence, an interior occurrence; and for a one-letter word, any
occurrence of the same variable.

Apply \cref{cor:one-index-dependence} to the chosen word of $\mathbf u$.
After restricting to the entries used at position $t$ of $q$, the indicator
of whether the selected entry is $a$ is either constant, depends only on its
row index, or depends only on its column index.  In particular, none of those
entries is $\infty$.

For each position of $q$, write its $a$-indicator in one of the forms
\[
 c_{t},\qquad \rho_{t}(s_{t-1}),\qquad \tau_{t}(s_{t}).
\]
Constant contributions, as well as contributions depending only on the fixed
endpoints $s_{0}=i$ or $s_{m}=j$, are collected in $C$.  At an internal state
$s_{t}$ there can be at most one contribution from the $t$-th position and at
most one from the $(t+1)$-st position.  Their sum defines $h_{t}(s_{t})$ and
belongs to $\{0,1,2\}$.  This gives \eqref{eq:separated-count}.
\end{proof}

\begin{proposition}\label{prop:three-index-word}
Let $n\geq3$, let $\mathbf u$ be a term and let $q$ be a word.  If
\[
 \mathbf u\approx\mathbf u+q
\]
fails in $\M_{n}(S_{7})$, then it fails in $\M_{3}(S_{7})$.
\end{proposition}

\begin{proof}
Choose an assignment and a coordinate $(i,j)$ witnessing the failure.  If
$\mathbf u_{ij}=\infty$, then
$(\mathbf u+q)_{ij}=\infty$ as well, so this cannot witness a failure.  Hence
\[
 \mathbf u_{ij}=c\in\{1,a\},\qquad q_{ij}\neq c.
\]

If one of the alternatives in \cref{lem:missing-position} occurs, the result
follows from that lemma.  We may therefore assume that none occurs.

Suppose first that $c=1$.  Every entry used by every index sequence of every
word in $\mathbf u$ is $1$.  For each occurrence in $q$, choose a corresponding
occurrence in $\mathbf u$ as in the proof of \cref{lem:separated-count}.  The
entire range of entries used at that position of $q$ is therefore equal to
$1$.  Thus every index sequence for $q$ has product $1$, giving $q_{ij}=1$, a
contradiction.  Hence the case $c=1$ cannot occur under the present
assumption.

Now let $c=a$.  By \cref{lem:separated-count}, all products associated with
$q$ use only $1$ and $a$, and their numbers of $a$'s have the form
\eqref{eq:separated-count}.  Since $q_{ij}\neq a$, there is an index sequence
$(k_{1},\dots,k_{m-1})$ for which
\[
 N(k_{1},\dots,k_{m-1})\neq1.
\]
If
\[
 N(i,\dots,i)\neq1,
\]
then the same failure already uses only the indices $i$ and $j$.
Otherwise, $N(i,\dots,i)=1$.  Equation \eqref{eq:separated-count} then implies
that for some $t$,
\[
 h_{t}(k_{t})\neq h_{t}(i),
\]
for otherwise the two displayed values of $N$ would coincide.  The index
sequence in which every internal index is $i$ except the $t$-th one, which is
$k_{t}$, has value
\[
 1-h_{t}(i)+h_{t}(k_{t})\neq1.
\]
Thus a product for $q$ that is not equal to $a$ uses only the indices in
\[
 D=\{i,j,k_{t}\}.
\]

Restrict every assigned matrix to the principal submatrix indexed by $D$.
Every index sequence for every word in $\mathbf u$ remains an index sequence
from the original assignment, so its product is still $a$.  The chosen index
sequence for $q$ is still present and is not $a$.  If $|D|<3$, enlarge $D$ to
a three-element subset of $[n]$; the same statements remain true.  Relabelling
that subset as $\{1,2,3\}$ gives a failure in $\M_{3}(S_{7})$.
\end{proof}

We next pass from the special form in \cref{prop:three-index-word} to arbitrary
identities.  The following elementary fact is useful.  If
$\mathbf s\approx\mathbf t$ is valid and $q$ is a word occurring as a summand
of $\mathbf t$, then
\[
 \mathbf s+q\approx\mathbf t+q\approx\mathbf t\approx\mathbf s,
\]
so $\mathbf s\approx\mathbf s+q$ is also valid.

\begin{theorem}\label{thm:dimension-three}
For every $n\geq3$,
\[
 \Id(\M_{n}(S_{7}))=\Id(\M_{3}(S_{7})).
\]
\end{theorem}

\begin{proof}
By \cref{cor:dimension-chain}, $\M_{3}(S_{7})$ embeds into
$\M_{n}(S_{7})$.  Therefore
\[
 \Id(\M_{n}(S_{7}))\subseteq\Id(\M_{3}(S_{7})).
\]

For the reverse inclusion, let $\mathbf s\approx\mathbf t$ be valid in
$\M_{3}(S_{7})$ and suppose, for a contradiction, that it fails in
$\M_{n}(S_{7})$.  At some coordinate the two values are distinct.  At least
one is in $\{1,a\}$; interchange the two sides if necessary and write
\[
 \mathbf s_{ij}=c\in\{1,a\},\qquad \mathbf t_{ij}\neq c.
\]
There is a word $q$ occurring in $\mathbf t$ such that $q_{ij}\neq c$;
otherwise every word of $\mathbf t$ would have value $c$, and so would their
sum.  Hence $\mathbf s\approx\mathbf s+q$ fails in $\M_{n}(S_{7})$.
By \cref{prop:three-index-word}, it fails in $\M_{3}(S_{7})$.

On the other hand, validity of $\mathbf s\approx\mathbf t$ in
$\M_{3}(S_{7})$ and the fact that $q$ is a summand of $\mathbf t$ imply the
validity of $\mathbf s\approx\mathbf s+q$ there, a contradiction.
\end{proof}

\section{The exact matrix-variety chain}\label{sec:exact-chain}

We now separate dimensions two and three.

\begin{theorem}\label{thm:epsilon}
The identity
\begin{equation}\label{eq:epsilon}
 x+y+xy+yx\ \approx\ x+y+xy+yx+x^{2}
\end{equation}
holds in $\M_{2}(S_{7})$ and fails in $\M_{3}(S_{7})$.
\end{theorem}

\begin{proof}
Let $A,B\in\M_{2}(S_{7})$ and put
\[
 C=A+B+AB+BA.
\]
Fix a coordinate $(i,j)$.  If $C_{ij}=\infty$, then
$(C+A^{2})_{ij}=\infty$.

Suppose that $C_{ij}=1$.  By \cref{lem:flat-sum},
\[
 A_{ij}=B_{ij}=(AB)_{ij}=(BA)_{ij}=1.
\]
The equality $(AB)_{ij}=1$ forces every entry in row $i$ of $A$ to be $1$,
and $(BA)_{ij}=1$ forces every entry in column $j$ of $A$ to be $1$.
Therefore every summand in $(A^{2})_{ij}$ is $1$, and
$(A^{2})_{ij}=1$.

Suppose finally that $C_{ij}=a$.  Then
\begin{equation}\label{eq:a-four}
 A_{ij}=B_{ij}=(AB)_{ij}=(BA)_{ij}=a.
\end{equation}
We must have $i\neq j$, since otherwise the summand
$A_{ii}B_{ii}=a^{2}=\infty$ occurs in $(AB)_{ii}$.  Since there are only two
indices, they are precisely $i$ and $j$.  In $(AB)_{ij}=a$, the summand with
intermediate index $i$ gives
\[
 A_{ii}B_{ij}=A_{ii}a=a,
\]
so $A_{ii}=1$.  In $(BA)_{ij}=a$, the summand with intermediate index $j$
gives
\[
 B_{ij}A_{jj}=aA_{jj}=a,
\]
so $A_{jj}=1$.  Consequently,
\[
 (A^{2})_{ij}=A_{ii}A_{ij}+A_{ij}A_{jj}=a+a=a.
\]
Thus \eqref{eq:epsilon} holds at every coordinate of $\M_{2}(S_{7})$.

For failure in dimension three, take
\[
 A=\begin{pmatrix}
 1&a&a\\
 1&1&1\\
 1&a&1
 \end{pmatrix},
 \qquad
 B=\begin{pmatrix}
 1&a&1\\
 1&1&1\\
 1&1&1
 \end{pmatrix}.
\]
At coordinate $(1,2)$,
\[
 A_{12}=B_{12}=a,
\]
\[
 (AB)_{12}=1a+a1+a1=a,
 \qquad
 (BA)_{12}=1a+a1+1a=a,
\]
whereas
\[
 (A^{2})_{12}=1a+a1+aa=\infty.
\]
The left side of \eqref{eq:epsilon} therefore has entry $a$ at $(1,2)$ and
the right side has entry $\infty$.
\end{proof}

\begin{corollary}[Exact chain]\label{cor:exact-chain}
For every $n\geq3$,
\[
 \V(S_{7})<\V(\M_{2}(S_{7}))<\V(\M_{3}(S_{7}))
 =\V(\M_{n}(S_{7})).
\]
\end{corollary}

\begin{proof}
The second strict inequality follows from \cref{thm:epsilon} and the embedding
$\M_{2}(S_{7})\hookrightarrow\M_{3}(S_{7})$.  The equality follows from
\cref{thm:dimension-three}.

The multiplication of $S_{7}$ is commutative, whereas that of
$\M_{2}(S_{7})$ is not.  For example, with
\[
 X=\begin{pmatrix}1&1\\1&1\end{pmatrix},\qquad
 Y=\begin{pmatrix}1&1\\1&a\end{pmatrix},
\]
one has
\[
 XY=\begin{pmatrix}1&\infty\\1&\infty\end{pmatrix}
 \neq
 \begin{pmatrix}1&1\\\infty&\infty\end{pmatrix}=YX.
\]
Thus $xy\approx yx$ separates $\V(S_{7})$ from
$\V(\M_{2}(S_{7}))$.
\end{proof}

\section{A complete coordinate criterion for identities}\label{sec:identity-criterion}

We now combine \cref{thm:dimension-three} with the known solution of the
equational problem for $S_{7}$.

Let $\mathbf u$ be a term.  Write $c(\mathbf u)$ for the set of variables
occurring in $\mathbf u$.  For a word $w$ and a variable $x$, let
$\Occ(x,w)$ denote the number of occurrences of $x$ in $w$.  Define $\delta(\mathbf u)$ to be the set of all nonempty subsets
$Z\subseteq c(\mathbf u)$ such that, for every word $w$ occurring in
$\mathbf u$, one has
\[
 Z\cap c(w)=\{x\}
 \quad\text{and}\quad
 \Occ(x,w)=1
\]
for some variable $x$.  Thus each word of $\mathbf u$ contains exactly one
letter from $Z$, and that letter occurs exactly once in the word.

The following criterion is due to Jackson, Ren and Zhao
\cite[Proposition 5.5]{JRZ}; it is also recorded in
\cite[Lemma 1.2]{GJRZ}.

\begin{theorem}[Equational criterion for $S_{7}$]\label{thm:S7-criterion}
For ai-semiring terms $\mathbf u,\mathbf v$,
\[
 S_{7}\models\mathbf u\approx\mathbf v
\]
if and only if
\[
 c(\mathbf u)=c(\mathbf v),\qquad
 \delta(\mathbf u)=\delta(\mathbf v).
\]
\end{theorem}

For each original variable $x$ and each $r,s\in[n]$, introduce a new variable
$x_{rs}$.  If $w=x_{1}\cdots x_{m}$ is a word, define its $(p,q)$-coordinate
expansion by
\begin{equation}\label{eq:coordinate-expansion-word}
 w_{pq}^{(n)}=
 \sum_{i_{1},\dots,i_{m-1}\in[n]}
 x_{1,p i_{1}}x_{2,i_{1}i_{2}}\cdots x_{m,i_{m-1}q}.
\end{equation}
For $m=1$, this means $w_{pq}^{(n)}=x_{pq}$.  If
$\mathbf u=u_{1}+\cdots+u_{r}$, put
\[
 \mathbf u_{pq}^{(n)}=(u_{1})_{pq}^{(n)}+\cdots+(u_{r})_{pq}^{(n)},
\]
with repetitions deleted.

\begin{proposition}\label{prop:coordinate-criterion-general}
Let $S$ be an ai-semiring.  Then
\[
 \M_{n}(S)\models\mathbf u\approx\mathbf v
\]
if and only if
\[
 S\models\mathbf u_{pq}^{(n)}\approx\mathbf v_{pq}^{(n)}
 \qquad\text{for all }p,q\in[n].
\]
\end{proposition}

\begin{proof}
Given a matrix assignment $x\mapsto X$, assign the scalar variable $x_{rs}$
the entry $X_{rs}$.  Formula \eqref{eq:path-expansion} shows, by induction on
words and then by addition, that the value of
$\mathbf u_{pq}^{(n)}$ is exactly the $(p,q)$ entry of the matrix value of
$\mathbf u$.  Conversely, an arbitrary scalar assignment to all variables
$x_{rs}$ uniquely determines a matrix assignment.  Thus no assignments are
lost in either direction.
\end{proof}

Simultaneous permutation of rows and columns is an automorphism of
$\M_{n}(S)$.  Consequently, all diagonal coordinate expansions are obtained
from one another by a renaming of variables, and the same is true of all
nondiagonal coordinate expansions.

\begin{theorem}[Stable identity criterion]\label{thm:stable-identity-criterion}
Let $n\geq3$.  Then
$\M_{n}(S_{7})\models\mathbf u\approx\mathbf v$ if and only if the following
four equalities hold:
\[
\begin{aligned}
 c(\mathbf u_{11}^{(3)})&=c(\mathbf v_{11}^{(3)}),&
 \delta(\mathbf u_{11}^{(3)})&=\delta(\mathbf v_{11}^{(3)}),\\
 c(\mathbf u_{12}^{(3)})&=c(\mathbf v_{12}^{(3)}),&
 \delta(\mathbf u_{12}^{(3)})&=\delta(\mathbf v_{12}^{(3)}).
\end{aligned}
\]
\end{theorem}

\begin{proof}
By \cref{thm:dimension-three}, it is enough to decide the identity in
$\M_{3}(S_{7})$.  By \cref{prop:coordinate-criterion-general}, this is
equivalent to validity in $S_{7}$ of all nine coordinate identities.  By
simultaneous permutation of the three indices, there are only two variable-
renaming types: a diagonal coordinate, represented by $(1,1)$, and a
nondiagonal coordinate, represented by $(1,2)$.  Apply
\cref{thm:S7-criterion} to these two scalar identities.
\end{proof}

\begin{remark}
\Cref{thm:stable-identity-criterion} gives a finite combinatorial decision
procedure for the identities of every $\M_{n}(S_{7})$ with $n\geq3$.  The
matrix dimension does not occur in the test after dimension three.
\end{remark}

\begin{corollary}\label{cor:uniform-decidability}
The equational theories of the matrix semirings $\M_{n}(S_{7})$, $n\geq3$,
are identical and admit a uniform effective decision procedure.
\end{corollary}

\begin{proof}
For a given identity, the two third-order coordinate expansions are finite
terms.  Their content sets and their $\delta$-families can be computed by a
finite search.  The conclusion therefore follows from
\cref{thm:stable-identity-criterion}.
\end{proof}

\section{Graph semirings inside \texorpdfstring{$\V(\M_{2}(S_{7}))$}{V(M2(S7))}}\label{sec:graph-semirings}

We recall the graph semirings of Gao and Ren \cite{GR}.  Let
$\mathbb G=(V,E)$ be a directed graph with no isolated vertices and with every
in-degree and out-degree at most one.  Let
\[
 S_{\mathbb G}=V\mathbin{\dot\cup}\{0,\omega\}.
\]
Its addition is flat, so the sum of two distinct elements is $0$, and its
multiplication is defined by
\[
 uv=\begin{cases}
 \omega,&(u,v)\in E,\\
 0,&(u,v)\notin E,
 \end{cases}
\]
for vertices $u,v$, while every product involving $0$ or $\omega$ is $0$.

Since in-degree and out-degree are at most one, each vertex has a partial
successor $s(u)$ and a partial predecessor $p(u)$.

\subsection{A two-by-two coding calculation}

For $b\in\{0,1\}$, put
\[
 e(b)=\begin{cases}1,&b=0,\\a,&b=1.\end{cases}
\]

\begin{lemma}\label{lem:bit-product}
Let $A,B\in\M_{2}(S_{7})$ satisfy $A_{21}=B_{21}=\infty$, and suppose
\[
 (A_{11},A_{12})=(e(r_{1}),e(r_{2})),\qquad
 (B_{12},B_{22})=(e(q_{1}),e(q_{2}))
\]
with $r_{1},r_{2},q_{1},q_{2}\in\{0,1\}$.  Then
\[
 (AB)_{12}=a
 \quad\Longleftrightarrow\quad
 (q_{1},q_{2})=(1-r_{1},1-r_{2}).
\]
Moreover, $(AB)_{12}=1$ if and only if all four bits are zero; in every other
case $(AB)_{12}=\infty$.
\end{lemma}

\begin{proof}
We have
\[
 (AB)_{12}=e(r_{1})e(q_{1})+e(r_{2})e(q_{2}).
\]
A product $e(r)e(q)$ is $1$ when $(r,q)=(0,0)$, is $a$ when exactly one of
$r,q$ is $1$, and is $\infty$ when $(r,q)=(1,1)$.  The sum is $a$ precisely
when both products are $a$, which is exactly the coordinatewise complement
condition.  The assertion about the value $1$ is equally immediate, and
flatness gives $\infty$ in all remaining cases.
\end{proof}

\subsection{The divisor construction}

Let
\[
 \Lambda=\{0,1\}^{V}
\]
be the set of all binary labellings of the vertices.  For
$\lambda\in\Lambda$ and $u\in V$, define
\[
 \alpha_{\lambda}(u)=
 \begin{cases}
  1-\lambda(s(u)),&s(u)\text{ is defined},\\
  1,&s(u)\text{ is undefined},
 \end{cases}
\]
and
\[
 \gamma_{\lambda}(u)=
 \begin{cases}
  1-\lambda(p(u)),&p(u)\text{ is defined},\\
  1,&p(u)\text{ is undefined}.
 \end{cases}
\]
Put
\begin{equation}\label{eq:Xulambda}
 X_{u}^{\lambda}=
 \begin{pmatrix}
  e(\alpha_{\lambda}(u))&e(\lambda(u))\\
  \infty&e(\gamma_{\lambda}(u))
 \end{pmatrix}
 \in\M_{2}(S_{7}),
\end{equation}
and define the tuple
\[
 X_{u}=(X_{u}^{\lambda})_{\lambda\in\Lambda}
 \in\M_{2}(S_{7})^{\Lambda}.
\]
Finally, put
\[
 \Omega=\begin{pmatrix}\infty&a\\\infty&\infty\end{pmatrix},
 \qquad
 \boldsymbol\Omega=(\Omega)_{\lambda\in\Lambda}.
\]

\begin{theorem}\label{thm:graph-divisor}
For every directed graph $\mathbb G$ as above,
\[
 S_{\mathbb G}\in\V(\M_{2}(S_{7})).
\]
More precisely, $S_{\mathbb G}$ is a homomorphic image of a subsemiring of
$\M_{2}(S_{7})^{\Lambda}$.
\end{theorem}

\begin{proof}
Let $H$ be the subsemiring of $\M_{2}(S_{7})^{\Lambda}$ generated by
\[
 \{X_{u}:u\in V\}\cup\{\boldsymbol\Omega\}.
\]

First suppose that $(u,v)\in E$.  Then $s(u)=v$ and $p(v)=u$.  For every
$\lambda\in\Lambda$, the bit pair in the first row of $X_{u}^{\lambda}$ is
\[
 (1-\lambda(v),\lambda(u)),
\]
while the bit pair in the second column of $X_{v}^{\lambda}$ is
\[
 (\lambda(v),1-\lambda(u)).
\]
These pairs are coordinatewise complements.  By
\cref{lem:bit-product}, the $(1,2)$ entry of
$X_{u}^{\lambda}X_{v}^{\lambda}$ is $a$.  All other entries of the product
are $\infty$, because both matrices have lower-left entry $\infty$.  Hence
\begin{equation}\label{eq:edge-product}
 X_{u}X_{v}=\boldsymbol\Omega\qquad((u,v)\in E).
\end{equation}

Now suppose that $(u,v)\notin E$.  We show that for some labelling $\lambda$,
\[
 (X_{u}^{\lambda}X_{v}^{\lambda})_{12}=\infty.
\]
There are three cases.

If $u$ has no successor, choose $\lambda(u)=\lambda(v)=1$.  The first row of
$X_{u}^{\lambda}$ is then $(a,a)$, and the upper entry in the second column of
$X_{v}^{\lambda}$ is $a$.  The first summand in the $(1,2)$ entry is
$a^{2}=\infty$.

If $s(u)=u$, then $v\neq u$.  Choose $\lambda(u)=1$ and $\lambda(v)=0$.
The first-row bit pair for $X_{u}^{\lambda}$ is $(0,1)$, whereas the first bit
in the second-column pair for $X_{v}^{\lambda}$ is $0$.  Thus the two pairs
are not complementary.  One summand in the $(1,2)$ entry is $1$; the other is
either $a$ or $\infty$.  In both cases their sum is $\infty$.

Finally, suppose $s(u)=w\neq u$.  Since $(u,v)$ is not an edge, $v\neq w$.
Choose
\[
 \lambda(u)=1,\qquad \lambda(w)=0,\qquad \lambda(v)=1.
\]
These prescriptions are consistent even when $v=u$.  The first row of
$X_{u}^{\lambda}$ is $(a,a)$, and the upper entry in the second column of
$X_{v}^{\lambda}$ is $a$, so again the first summand is $\infty$.

Define\[
 J=\{Y=(Y^{\lambda})_{\lambda\in\Lambda}\in H:
       (Y^{\lambda})_{12}=\infty\text{ for some }\lambda\in\Lambda\}.
\]
Every generator, and hence every element of $H$, has lower-left entry
$\infty$ in every coordinate.  If $Y\in J$ and $Z\in H$, then at a coordinate
where $Y_{12}^{\lambda}=\infty$ one has
\[
 (Y^{\lambda}+Z^{\lambda})_{12}=\infty,
\]
\[
 (Y^{\lambda}Z^{\lambda})_{12}=\infty,
 \qquad
 (Z^{\lambda}Y^{\lambda})_{12}=\infty.
\]
Thus $J$ is closed upward under addition and is a two-sided multiplicative
ideal.  The equivalence relation that identifies all elements of $J$ and
leaves all other elements as singleton classes is therefore a semiring
congruence; denote it by $\theta$.

The preceding calculations give
\[
 X_{u}X_{v}=\begin{cases}
 \boldsymbol\Omega,&(u,v)\in E,\\
 \text{an element of }J,&(u,v)\notin E.
 \end{cases}
\]
Furthermore,
\[
 \boldsymbol\Omega X_{u},\quad X_{u}\boldsymbol\Omega,
 \quad\boldsymbol\Omega^{2}\in J.
\]
If $u\neq v$, choose a labelling with $\lambda(u)=0$ and $\lambda(v)=1$.
Then the $(1,2)$ entries of $X_{u}^{\lambda}$ and $X_{v}^{\lambda}$ are
$1$ and $a$, so $X_{u}+X_{v}\in J$.  Likewise,
$X_{u}+\boldsymbol\Omega\in J$ by choosing $\lambda(u)=0$.

It remains only to check that the quotient has no further classes.  By
distributivity, every element of $H$ is a finite sum of multiplicative words
in the generators.  A word of length one is an $X_{u}$ or
$\boldsymbol\Omega$.  A word of length two is either
$\boldsymbol\Omega$ or belongs to $J$, and every word of length at least three
belongs to $J$.  A sum outside $J$ can therefore contain only repeated copies
of one and the same element among the $X_{u}$ and
$\boldsymbol\Omega$; any two distinct such elements have sum in $J$.
Consequently,
\[
 H/\theta=\{J,\boldsymbol\Omega/\theta,X_{u}/\theta:u\in V\}.
\]
The map
\[
 0\mapsto J,\qquad
 \omega\mapsto\boldsymbol\Omega/\theta,\qquad
 u\mapsto X_{u}/\theta
\]
is now visibly an isomorphism $S_{\mathbb G}\to H/\theta$.
\end{proof}

Gao and Ren proved that every nontrivial subdirectly irreducible
$3$-nilpotent flat semiring is a graph semiring
\cite[Theorem 2.1]{GR}.  Since $\NF_{3}$ is a subvariety of the variety of
flat semirings and its subdirectly irreducible members are $3$-nilpotent, the
subdirect representation theorem gives the following consequence.

\begin{corollary}\label{cor:NF3}
\[
 \NF_{3}\leq\V(\M_{2}(S_{7})).
\]
\end{corollary}

\begin{proof}
By the subdirect representation theorem (see, for example, \cite{BS}), every
algebra in $\NF_{3}$ is a subdirect product of subdirectly irreducible
members of $\NF_{3}$.  Each nontrivial factor is a graph semiring by
\cite[Theorem 2.1]{GR}, and every graph semiring lies in
$\V(\M_{2}(S_{7}))$ by \cref{thm:graph-divisor}.  The trivial factor also lies
there.  Closure under direct products and subalgebras completes the proof.
\end{proof}

\section{Continuum many intermediate varieties}\label{sec:continuum}

For $m\geq2$, define the directed-cycle term
\[
 \mathbf c_{m}=x_{1}x_{2}+x_{2}x_{3}+\cdots+x_{m-1}x_{m}+x_{m}x_{1}
\]
and its reverse
\[
 \mathbf c_{m}^{\op}
 =x_{2}x_{1}+x_{3}x_{2}+\cdots+x_{m}x_{m-1}+x_{1}x_{m}.
\]
Let $S_{\mathbf c_{q}}$ be the graph semiring of the directed $q$-cycle
\[
 a_{1}\to a_{2}\to\cdots\to a_{q}\to a_{1}.
\]

\begin{lemma}\label{lem:cycle-reversal}
Let $q\geq3$ and $m\geq2$.  Then
\[
 S_{\mathbf c_{q}}\models
 \mathbf c_{m}\approx\mathbf c_{m}^{\op}
 \quad\Longleftrightarrow\quad q\nmid m.
\]
\end{lemma}

\begin{proof}
In a graph semiring, every product of two elements belongs to
$\{0,\omega\}$.  The value of $\mathbf c_{m}$ is $\omega$ precisely when all
its summands are $\omega$, which means that the assigned vertices form a
closed directed walk of length $m$ around the $q$-cycle.  Such a walk exists
if and only if $q\mid m$.  If $q\nmid m$, the value of $\mathbf c_{m}$ is
therefore $0$ under every assignment.  The same argument, with the direction
reversed, applies to $\mathbf c_{m}^{\op}$, so the identity holds.

Now suppose $q\mid m$.  Assign the variables successively around the directed
cycle, periodically with period $q$.  Then every summand of $\mathbf c_{m}$ is
$\omega$, and hence $\mathbf c_{m}=\omega$.  Since $q\geq3$, no reversed
consecutive pair is a directed edge; every summand of
$\mathbf c_{m}^{\op}$ is $0$.  Thus the identity fails.
\end{proof}

Let $P$ be the set of odd primes.  For $p\in P$, write
\begin{equation}\label{eq:sigma-p}
 \sigma_{p}:\qquad \mathbf c_{p}\approx\mathbf c_{p}^{\op}.
\end{equation}
Since multiplication in $S_{7}$ is commutative,
\begin{equation}\label{eq:S7-sigma}
 S_{7}\models\sigma_{p}\qquad(p\in P).
\end{equation}
By \cref{lem:cycle-reversal}, for odd primes $p,q$,
\begin{equation}\label{eq:cycle-sigma}
 S_{\mathbf c_{q}}\models\sigma_{p}
 \quad\Longleftrightarrow\quad p\neq q.
\end{equation}

For each $Q\subseteq P$, define
\begin{equation}\label{eq:WQ}
 \mathcal W_{Q}
 =\V\bigl(S_{7},\{S_{\mathbf c_{q}}:q\in Q\}\bigr).
\end{equation}
By \cref{thm:graph-divisor}, every generator belongs to
$\V(\M_{2}(S_{7}))$, and hence
\[
 \V(S_{7})\leq\mathcal W_{Q}\leq\V(\M_{2}(S_{7})).
\]
Equations \eqref{eq:S7-sigma} and \eqref{eq:cycle-sigma} give
\begin{equation}\label{eq:WQ-sigma}
 \mathcal W_{Q}\models\sigma_{p}
 \quad\Longleftrightarrow\quad p\notin Q.
\end{equation}

\begin{theorem}\label{thm:power-set-embedding}
The map
\[
 Q\longmapsto\mathcal W_{Q}
\]
is an order embedding of the power set $\mathcal P(P)$ into the interval
\[
 [\V(S_{7}),\V(\M_{2}(S_{7}))].
\]
\end{theorem}

\begin{proof}
If $Q\subseteq R$, the generating set for $\mathcal W_{Q}$ is contained in
the generating set for $\mathcal W_{R}$, so
$\mathcal W_{Q}\leq\mathcal W_{R}$.

Conversely, suppose $Q\nsubseteq R$ and choose $p\in Q\setminus R$.  By
\eqref{eq:WQ-sigma}, the variety $\mathcal W_{R}$ satisfies $\sigma_{p}$,
whereas $\mathcal W_{Q}$ does not, because it contains
$S_{\mathbf c_{p}}$.  Hence $\mathcal W_{Q}\nleq\mathcal W_{R}$.
\end{proof}

\begin{theorem}\label{thm:continuum-interval}
For every $n\geq2$,
\[
 \bigl|[\V(S_{7}),\V(\M_{n}(S_{7}))]\bigr|=2^{\aleph_{0}}.
\]
Each interval contains a chain and an antichain of cardinality
$2^{\aleph_{0}}$.
\end{theorem}

\begin{proof}
The set $P$ is countably infinite, so $\mathcal P(P)$ has cardinality
$2^{\aleph_{0}}$ and has both a chain and an antichain of that cardinality.
By \cref{thm:power-set-embedding}, these occur already in
$[\V(S_{7}),\V(\M_{2}(S_{7}))]$.  By
\cref{cor:dimension-chain}, this interval is contained in the corresponding
interval for every $n\geq2$.

There are only countably many identities in the countable ai-semiring
language, so the class of all ai-semiring varieties has cardinality at most
$2^{\aleph_{0}}$.  The lower and upper bounds therefore coincide.
\end{proof}

Jiao and Ren proved that every variety in
$[\V(S_{7}),\V(\M_{n}(S_{7}))]$ is nonfinitely based
\cite[Corollary 3.8]{JR}.  We immediately obtain the promised strengthening
of their countability result.

\begin{corollary}\label{cor:continuum-NFB}
For every $n\geq2$, the interval
$[\V(S_{7}),\V(\M_{n}(S_{7}))]$ contains continuum many nonfinitely based
varieties, including a chain and an antichain of cardinality
$2^{\aleph_{0}}$.
\end{corollary}

\section{The last three powers of \texorpdfstring{$\M_{n}(S_{7})\setminus\{[1]_{n}\}$}{Mn(S7) minus [1]n}}\label{sec:terminal-powers}

Put
\[
 N_{n}=\M_{n}(S_{7})\setminus\{[1]_{n}\}.
\]
This is a subsemiring of $\M_{n}(S_{7})$ \cite{JR}.  For nonempty
$I,J\subseteq[n]$, define
\begin{equation}\label{eq:EIJ}
 (E_{I,J})_{rs}=
 \begin{cases}
  a,&r\in I,\ s\in J,\\
  \infty,&\text{otherwise}.
 \end{cases}
\end{equation}

\begin{theorem}\label{thm:terminal-powers}
For every $n\geq2$,
\begin{align}
 N_{n}^{3}
 &=\{[\infty]_{n}\}\cup
   \{E_{I,J}:\varnothing\neq I,J\subseteq[n],
                 (I,J)\neq([n],[n])\},\label{eq:N3}\\
 N_{n}^{4}
 &=\{[\infty]_{n}\}\cup
   \{E_{I,J}:\varnothing\neq I,J\subsetneq[n]\},\label{eq:N4}\\
 N_{n}^{5}&=\{[\infty]_{n}\}.\label{eq:N5}
\end{align}
Consequently,
\[
 |N_{n}^{3}|=(2^{n}-1)^{2},\qquad
 |N_{n}^{4}|=(2^{n}-2)^{2}+1.
\]
In particular, the multiplicative reduct of $N_{n}$ is $5$-nilpotent but not
$4$-nilpotent.
\end{theorem}

\begin{proof}
We first determine $N_{n}^{3}$.  Let $A,B,C\in N_{n}$.  If
$(ABC)_{ij}=1$, then every product
\[
 A_{ir}B_{rs}C_{sj}\qquad(r,s\in[n])
\]
is $1$.  Hence $B_{rs}=1$ for all $r,s$, contradicting $B\in N_{n}$.
Thus every entry of $ABC$ belongs to $\{a,\infty\}$.

Assume that $(ABC)_{ij}=a$.  Then
\begin{equation}\label{eq:ABC-all-a}
 A_{ir}B_{rs}C_{sj}=a\qquad(r,s\in[n]).
\end{equation}
All entries in \eqref{eq:ABC-all-a} lie in $\{1,a\}$, and each product
contains exactly one $a$.  Define bits
\[
 \alpha_{r}=\begin{cases}0,&A_{ir}=1,\\1,&A_{ir}=a,\end{cases}
 \quad
 \beta_{rs}=\begin{cases}0,&B_{rs}=1,\\1,&B_{rs}=a,\end{cases}
 \quad
 \gamma_{s}=\begin{cases}0,&C_{sj}=1,\\1,&C_{sj}=a.\end{cases}
\]
Then
\begin{equation}\label{eq:binary-ABC}
 \alpha_{r}+\beta_{rs}+\gamma_{s}=1
 \qquad(r,s\in[n]).
\end{equation}

Suppose first that $\alpha_{r_{0}}=1$ for some $r_{0}$.  Equation
\eqref{eq:binary-ABC} gives $\gamma_{s}=0$ for every $s$, and then
\[
 \beta_{rs}=1-\alpha_{r}\qquad(r,s\in[n]).
\]
Let
\[
 K=\{r:\alpha_{r}=0\}.
\]
The set $K$ is nonempty, since otherwise $B=[1]_{n}$.  Thus the $r$-th row of
$B$ is constantly $a$ for $r\in K$ and constantly $1$ for $r\notin K$.
Define
\[
 I=\{p:\ A_{pr}=1\ (r\in K),\ A_{pr}=a\ (r\notin K)\},
\]
\[
 J=\{q:\ C_{sq}=1\text{ for every }s\in[n]\}.
\]
We claim that
\begin{equation}\label{eq:rectangle-case1}
 (ABC)_{pq}=a\quad\Longleftrightarrow\quad p\in I,\ q\in J.
\end{equation}
If $p\in I$ and $q\in J$, every product
$A_{pr}B_{rs}C_{sq}$ contains exactly one $a$, so the coordinate is $a$.
Conversely, if the coordinate is $a$, then every such product is $a$.  Taking
$r\in K$, where $B_{rs}=a$, forces $A_{pr}=1$ and $C_{sq}=1$ for all $s$;
hence $q\in J$.  Taking $r\notin K$ then forces $A_{pr}=a$.  Thus $p\in I$.
The original indices satisfy $i\in I$ and $j\in J$.  Moreover, $J\neq[n]$,
for otherwise $C=[1]_{n}$.  Hence
$ABC=E_{I,J}$ with $I,J$ nonempty and $J$ proper.

Suppose next that every $\alpha_{r}=0$.  Equation
\eqref{eq:binary-ABC} becomes
\[
 \beta_{rs}=1-\gamma_{s}.
\]
Let
\[
 K=\{s:\gamma_{s}=0\}.
\]
Again $K$ is nonempty, since otherwise $B=[1]_{n}$.  Now the $s$-th column of
$B$ is constantly $a$ for $s\in K$ and constantly $1$ for $s\notin K$.
Define
\[
 I=\{p:\ A_{pr}=1\text{ for every }r\in[n]\},
\]
\[
 J=\{q:\ C_{sq}=1\ (s\in K),\ C_{sq}=a\ (s\notin K)\}.
\]
The same coordinatewise argument gives
\[
 (ABC)_{pq}=a\quad\Longleftrightarrow\quad p\in I,\ q\in J.
\]
Here $I$ is proper, because $A\neq[1]_{n}$.  Thus every nonzero triple product
has the form claimed in \eqref{eq:N3}.

Conversely, let $I,J$ be nonempty and not both equal to $[n]$.  If both are
proper, set
\[
 A_{pr}=\begin{cases}1,&p\in I,\\\infty,&p\notin I,\end{cases}
 \qquad B=[a]_{n},
 \qquad
 C_{sq}=\begin{cases}1,&q\in J,\\\infty,&q\notin J.\end{cases}
\]
Then $A,B,C\in N_{n}$ and $ABC=E_{I,J}$.

If $I=[n]$ and $J\subsetneq[n]$, choose a nonempty proper set
$K\subsetneq[n]$ and put
\[
 A_{pr}=\begin{cases}1,&r\in K,\\a,&r\notin K,\end{cases}
 \qquad
 B_{rs}=\begin{cases}a,&r\in K,\\1,&r\notin K,\end{cases}
\]
with $C$ the column selector used above.  Then $ABC=E_{[n],J}$.
The case $I\subsetneq[n]$ and $J=[n]$ is symmetric: take the row selector
$A$, and put
\[
 B_{rs}=\begin{cases}1,&s\in K,\\a,&s\notin K,\end{cases}
 \qquad
 C_{sq}=\begin{cases}a,&s\in K,\\1,&s\notin K.\end{cases}
\]
This proves \eqref{eq:N3}.

We next determine $N_{n}^{4}$.  Let $E_{I,J}$ be a nonzero element of
$N_{n}^{3}$ and let $D\in N_{n}$.  If $J\subsetneq[n]$, then every row of
$E_{I,J}$ contains $\infty$, so
\[
 E_{I,J}D=[\infty]_{n}.
\]
If $J=[n]$, then $I\subsetneq[n]$.  Define
\[
 J(D)=\{q:\ D_{kq}=1\text{ for every }k\in[n]\}.
\]
For $p\in I$,
\[
 (E_{I,[n]}D)_{pq}=\sum_{k=1}^{n}aD_{kq},
\]
which is $a$ precisely when $q\in J(D)$ and is $\infty$ otherwise.  For
$p\notin I$ it is always $\infty$.  Thus
\[
 E_{I,[n]}D=
 \begin{cases}
 E_{I,J(D)},&J(D)\neq\varnothing,\\
 [\infty]_{n},&J(D)=\varnothing.
 \end{cases}
\]
Since $D\neq[1]_{n}$, the set $J(D)$ is proper.  Hence every nonzero element
of $N_{n}^{4}$ has the form in \eqref{eq:N4}.

Conversely, if $I,J$ are nonempty proper subsets, then
$E_{I,[n]}\in N_{n}^{3}$.  Taking
\[
 D_{kq}=\begin{cases}1,&q\in J,\\\infty,&q\notin J\end{cases}
\]
gives $E_{I,[n]}D=E_{I,J}$.  This proves \eqref{eq:N4}.

Finally, every $E_{I,J}$ occurring in $N_{n}^{4}$ has $J\subsetneq[n]$, so
every row contains $\infty$ and its product on the right by any matrix is
$[\infty]_{n}$.  Hence \eqref{eq:N5} holds.

There are $(2^{n}-1)^{2}-1$ admissible nonzero rectangles in
\eqref{eq:N3}; after adding $[\infty]_{n}$ this gives
$(2^{n}-1)^{2}$.  There are $(2^{n}-2)^{2}$ rectangles in
\eqref{eq:N4}, giving the second cardinality formula.  Since
$N_{n}^{4}$ contains nonzero matrices, the nilpotency index is exactly five.
\end{proof}

\section{Matrix operators on the variety lattice}\label{sec:matrix-operators}

This final section records general consequences of the matrix construction.
Let $\mathcal L_{\mathrm{ai}}$ be the lattice of all ai-semiring varieties.
For $n\geq1$, define
\[
 \mathscr M_{n}(\mathcal V)
 =\V\{\M_{n}(S):S\in\mathcal V\}.
\]
The matrix construction preserves subalgebras, homomorphic images and direct
products; the direct-product assertion follows from the coordinatewise
isomorphism
\[
 \M_{n}\left(\prod_{i\in I}S_{i}\right)
 \cong\prod_{i\in I}\M_{n}(S_{i}).
\]

\begin{theorem}\label{thm:matrix-operators}
For ai-semiring varieties $\mathcal V$ and $(\mathcal V_{i})_{i\in I}$ and
positive integers $m,n$, the following hold.
\begin{enumerate}[label=\textup{(\roman*)}]
\item $\mathscr M_{1}(\mathcal V)=\mathcal V$;
\item $\mathscr M_{m}(\mathscr M_{n}(\mathcal V))
      =\mathscr M_{mn}(\mathcal V)$;
\item
\[
 \mathscr M_{n}\left(\bigvee_{i\in I}\mathcal V_{i}\right)
 =\bigvee_{i\in I}\mathscr M_{n}(\mathcal V_{i});
\]
\item $\mathcal V\leq\mathscr M_{n}(\mathcal V)$;
\item if $n\leq m$, then
$\mathscr M_{n}(\mathcal V)\leq\mathscr M_{m}(\mathcal V)$.
\end{enumerate}
\end{theorem}

\begin{proof}
Part \textup{(i)} is immediate.  For \textup{(ii)}, the usual block
identification gives a natural isomorphism
\[
 \M_{m}(\M_{n}(S))\cong\M_{mn}(S).
\]
Because the matrix construction passes through H, S and P, applying
$\mathscr M_{m}$ to the variety generated by all $\M_{n}(S)$ produces exactly
the variety generated by the corresponding $\M_{mn}(S)$.

For \textup{(iii)}, write the join as an HSP-closure of the union of the
$\mathcal V_{i}$ and again pass the matrix construction through H, S and P.
Part \textup{(iv)} follows from the constant-matrix embedding
$s\mapsto[s]_{n}$.  Part \textup{(v)} follows from
\cref{thm:surjective-embedding}.
\end{proof}

Call a variety $\mathcal V$ \emph{matrix-stable} if
$\mathscr M_{n}(\mathcal V)=\mathcal V$ for every $n\geq1$.

\begin{theorem}[Two-by-two test and stable closure]\label{thm:matrix-stable}
For an ai-semiring variety $\mathcal V$, the following are equivalent:
\[
 \mathscr M_{2}(\mathcal V)=\mathcal V,
 \qquad
 \mathscr M_{n}(\mathcal V)=\mathcal V\quad(n\geq1).
\]
Moreover,
\[
 \operatorname{Mat}(\mathcal V)
 =\bigvee_{k\geq0}\mathscr M_{2^{k}}(\mathcal V)
\]
is the least matrix-stable variety containing $\mathcal V$.
\end{theorem}

\begin{proof}
Assume $\mathscr M_{2}(\mathcal V)=\mathcal V$.  By
\cref{thm:matrix-operators}(ii),
\[
 \mathscr M_{2^{k}}(\mathcal V)=\mathcal V
\]
for every $k$.  Given $n$, choose $k$ with $n\leq2^{k}$.  Parts
\textup{(iv)} and \textup{(v)} of \cref{thm:matrix-operators} give
\[
 \mathcal V\leq\mathscr M_{n}(\mathcal V)
 \leq\mathscr M_{2^{k}}(\mathcal V)=\mathcal V.
\]
The converse is immediate.

Let $\mathcal W=\operatorname{Mat}(\mathcal V)$.  By join preservation and
composition,
\[
 \mathscr M_{2}(\mathcal W)
 =\bigvee_{k\geq0}\mathscr M_{2^{k+1}}(\mathcal V)
 \leq\mathcal W.
\]
The reverse inclusion follows from extensivity, so $\mathcal W$ is
matrix-stable.  If a matrix-stable $\mathcal U$ contains $\mathcal V$, then
\[
 \mathscr M_{2^{k}}(\mathcal V)
 \leq\mathscr M_{2^{k}}(\mathcal U)=\mathcal U
\]
for every $k$, and hence $\mathcal W\leq\mathcal U$.
\end{proof}

For a variety $\mathcal W$, define
\[
 \mathscr R_{n}(\mathcal W)
 =\{S:\M_{n}(S)\in\mathcal W\}.
\]
The class on the right is itself a variety.  Indeed, the matrix construction
preserves subalgebras, homomorphic images, and direct products.

\begin{theorem}[Right adjoints and the matrix-stable core]
\label{thm:right-adjoint-core}
For every positive integer $n$, the operator $\mathscr R_{n}$ is right
adjoint to $\mathscr M_{n}$; explicitly,
\[
 \mathscr M_{n}(\mathcal V)\leq\mathcal W
 \quad\Longleftrightarrow\quad
 \mathcal V\leq\mathscr R_{n}(\mathcal W).
\]
Moreover,
\[
 \mathscr R_{m}\mathscr R_{n}=\mathscr R_{mn},
 \qquad
 \mathscr R_{n}(\mathcal W)\leq\mathcal W.
\]
The variety
\[
 \Core(\mathcal W)
 =\bigcap_{k\geq0}\mathscr R_{2^{k}}(\mathcal W)
\]
is the greatest matrix-stable subvariety of $\mathcal W$.
\end{theorem}

\begin{proof}
Suppose first that $\mathscr M_{n}(\mathcal V)\leq\mathcal W$.  For every
$S\in\mathcal V$, the algebra $\M_{n}(S)$ belongs to $\mathcal W$, whence
$S\in\mathscr R_{n}(\mathcal W)$.  Thus
$\mathcal V\leq\mathscr R_{n}(\mathcal W)$.  Conversely, if
$\mathcal V\leq\mathscr R_{n}(\mathcal W)$, then all generators
$\M_{n}(S)$ with $S\in\mathcal V$ lie in $\mathcal W$, and therefore
$\mathscr M_{n}(\mathcal V)\leq\mathcal W$.

The natural block isomorphism
$\M_{m}(\M_{n}(S))\cong\M_{mn}(S)$ gives
$\mathscr R_{m}\mathscr R_{n}=\mathscr R_{mn}$.  If
$S\in\mathscr R_{n}(\mathcal W)$, then $\M_{n}(S)\in\mathcal W$; the
constant-matrix embedding $S\hookrightarrow\M_{n}(S)$ shows that
$S\in\mathcal W$.  Hence $\mathscr R_{n}(\mathcal W)\leq\mathcal W$.

Put $\mathcal C=\Core(\mathcal W)$.  Clearly $\mathcal C\leq\mathcal W$.
For every $k\geq0$, the inclusion
$\mathcal C\leq\mathscr R_{2^{k+1}}(\mathcal W)$ and the adjunction imply
\[
 \mathscr M_{2^{k+1}}(\mathcal C)\leq\mathcal W.
\]
Using composition of the matrix operators, this is equivalent to
\[
 \mathscr M_{2^{k}}(\mathscr M_{2}(\mathcal C))\leq\mathcal W,
\]
and the adjunction again yields
$\mathscr M_{2}(\mathcal C)\leq\mathscr R_{2^{k}}(\mathcal W)$ for every
$k$.  Thus $\mathscr M_{2}(\mathcal C)\leq\mathcal C$.  The reverse
inclusion follows from extensivity, and hence $\mathcal C$ is matrix-stable
by \cref{thm:matrix-stable}.

Finally, let $\mathcal U$ be matrix-stable and suppose
$\mathcal U\leq\mathcal W$.  Then
$\mathscr M_{2^{k}}(\mathcal U)=\mathcal U\leq\mathcal W$ for all $k$.
By adjunction, $\mathcal U\leq\mathscr R_{2^{k}}(\mathcal W)$ for all $k$,
so $\mathcal U\leq\Core(\mathcal W)$.  This proves maximality.
\end{proof}

\begin{corollary}\label{cor:stable-lattice}
The matrix-stable varieties form a complete sublattice of the lattice of all
ai-semiring varieties.  For every variety $\mathcal V$,
\[
 \Core(\mathcal V)\leq\mathcal V\leq\Mat(\mathcal V),
\]
where $\Mat$ is a completely join-preserving closure operator and $\Core$ is
an interior operator.  The former is the least matrix-stable variety above
$\mathcal V$, and the latter is the greatest matrix-stable variety below it.
\end{corollary}

\begin{proof}
The assertions about $\Mat$ and $\Core$ follow from
\cref{thm:matrix-stable,thm:right-adjoint-core}.  Arbitrary joins of
matrix-stable varieties are matrix-stable because $\mathscr M_{2}$ preserves
joins.  If $(\mathcal V_{i})_{i\in I}$ are matrix-stable, monotonicity gives
\[
 \mathscr M_{2}\Bigl(\bigcap_{i\in I}\mathcal V_{i}\Bigr)
 \leq\bigcap_{i\in I}\mathscr M_{2}(\mathcal V_{i})
 =\bigcap_{i\in I}\mathcal V_{i},
\]
while extensivity gives the reverse inclusion.  Hence arbitrary meets are
matrix-stable as well.
\end{proof}

The next result gives a general restriction on the possible behaviour of a
matrix-variety chain.

\begin{theorem}[Propagation of an equality]\label{thm:equality-propagation}
Let
\[
 \mathcal V_{r}=\mathscr M_{r}(\mathcal V)\qquad(r\geq1).
\]
If $\mathcal V_{p}=\mathcal V_{q}$ for some $1\leq p<q$, then the chain
$(\mathcal V_{r})_{r\geq1}$ is eventually constant.  More precisely, if
\[
 r_{0}=\left\lceil\frac{p}{q-p}\right\rceil,
\]
then
\[
 \mathcal V_{s}=\mathcal V_{t}
 \qquad(s,t\geq r_{0}p).
\]
\end{theorem}

\begin{proof}
For every positive integer $r$, apply $\mathscr M_{r}$ to
$\mathcal V_{p}=\mathcal V_{q}$.  By
\cref{thm:matrix-operators}(ii),
\[
 \mathcal V_{rp}=\mathcal V_{rq}.
\]
The chain is increasing, so it is constant on every integer interval
$[rp,rq]$.  For $r\geq r_{0}$,
\[
 (r+1)p\leq rq,
\]
so the intervals $[rp,rq]$ and $[(r+1)p,(r+1)q]$ overlap.  Their union, for
$r\geq r_{0}$, contains every integer at least $r_{0}p$.  Since consecutive
intervals share an index, the constant values on all of them are equal.
\end{proof}

\begin{corollary}\label{cor:dichotomy}
For every ai-semiring variety $\mathcal V$, the chain
\[
 \mathcal V\leq\mathscr M_{2}(\mathcal V)
 \leq\mathscr M_{3}(\mathcal V)\leq\cdots
\]
is either strictly increasing at every step or eventually constant.
Moreover, it is constant from $p$ onward if and only if
\[
 \mathscr M_{p}(\mathcal V)=\mathscr M_{2p}(\mathcal V).
\]
\end{corollary}

\begin{proof}
If any equality occurs, \cref{thm:equality-propagation} gives eventual
constancy.  If no equality occurs, every step is strict.  For the final
assertion, necessity is clear.  If the displayed equality holds, apply
\cref{thm:equality-propagation} with $q=2p$; then $r_{0}=1$.
\end{proof}

\end{document}